\documentclass{amsart}
\usepackage{amsmath}
\usepackage{comment}
\usepackage{amssymb}
\usepackage{amsthm}

\newcommand{\theoremref}[1]{\hyperref[#1]{Theorem~\ref*{#1}}}
\newcommand{\claimref}[1]{\hyperref[#1]{Claim~\ref*{#1}}}
\newcommand{\situationref}[1]{\hyperref[#1]{Situation~\ref*{#1}}}
\newcommand{\lemmaref}[1]{\hyperref[#1]{Lemma~\ref*{#1}}}
\newcommand{\definitionref}[1]{\hyperref[#1]{Definition~\ref*{#1}}}
\newcommand{\propositionref}[1]{\hyperref[#1]{Proposition~\ref*{#1}}}
\newcommand{\conjectureref}[1]{\hyperref[#1]{Conjecture~\ref*{#1}}}
\newcommand{\corollaryref}[1]{\hyperref[#1]{Corollary~\ref*{#1}}}
\newcommand{\exerciseref}[1]{\hyperref[#1]{Exercise~\ref*{#1}}}

\usepackage[mathscr]{euscript}

\usepackage[all]{xy}

\xyoption{2cell}
\UseAllTwocells

\usepackage{multicol}

\usepackage{xr-hyper}

\usepackage{hyperref}
\usepackage[usenames,dvipsnames]{xcolor}
\hypersetup{colorlinks=true,citecolor=NavyBlue,linkcolor=BrickRed,urlcolor=Orange}

\numberwithin{equation}{section}

\theoremstyle{plain}
\newtheorem{theorem}[equation]{Theorem}
\newtheorem{proposition}[equation]{Proposition}
\newtheorem{lemma}[equation]{Lemma}

\newtheorem{corollary}[equation]{Corollary}

\newtheorem{conjecture}[equation]{Conjecture}

\theoremstyle{definition}
\newtheorem{definition}[equation]{Definition}

\theoremstyle{remark}
\newtheorem{remark}[equation]{Remark}

\usepackage{tikz}
\usepackage{tikz-cd}
\usetikzlibrary{matrix,arrows}
\usetikzlibrary{decorations.pathmorphing}

\newcommand{\R}{\mathbb{R}}
\newcommand{\Q}{\mathbb{Q}}
\newcommand{\Z}{\mathbb{Z}}
\newcommand{\A}{\mathbb{A}}

\newcommand{\C}{\mathbb{C}}
\renewcommand{\P}{\mathbb{P}}

\DeclareMathOperator{\Spec}{Spec}

\DeclareMathOperator{\im}{Im}

\DeclareMathOperator{\Gr}{Gr}
\title{The $P=W$ identity for cluster varieties}
\author{Zili Zhang}
\address{University of Michigan, Department of Mathematics}
\email{ziliz@umich.edu}
\begin{document}
\maketitle

\setcounter{tocdepth}{1} 

\begin{abstract}
We find new examples of the $P=W$ identities of de Cataldo-Hausel-Migliorini by studying cluster varieties. We prove that the weight filtration of 2D cluster varieties correspond to the perverse filtration of elliptic fibrations which are deformation equivalent to the elliptic fibrations of types $I_b$ or $2I_b$. These examples do not arise from character varieties or moduli of Higgs bundles.

\end{abstract}
\tableofcontents

\section{Introduction}
\subsection{Perverse filtrations}
Let $Y$ be a K\"ahler manifold. The bounded derived category of constructible sheaves $D^b_c(Y)$ is naturally endowed with a perverse $t$-structure. Denote $^{\mathfrak{p}}\tau_{\le k}$ the perverse truncation functor. There is a natural morphism
\[
^{\mathfrak{p}}\tau_{\le k}K\rightarrow K
\]
for any object $K\in D^b_c(Y)$. Let $h:X\to Y$ be a proper morphism between K\"ahler manifolds. By taking the cohomology functor, the natural morphism
\[
^{\mathfrak{p}}\tau_{\le k} Rh_\ast\Q_X \to Rh_\ast\Q_X
\]
defines a map between the cohomology groups
\begin{equation}\label{0000}
\mathbb{H}^{d-\dim X+r}\left(Y,{^\mathfrak{p}}\tau_{\le k} (Rh_\ast\Q_X[\dim X-r])\right)\rightarrow H^d(X;\Q),
\end{equation}
where $r=\dim X\times_Y X-\dim X$ is the defect of the morphism $f:X\to Y$. We define $P_kH^d(X;\Q)\subset H^d(X;\Q)$ to be the image of (\ref{0000}), and we call the increasing filtration
\[
P_0H^\ast(X)\subset P_1H^\ast(X)\subset\cdots\subset P_kH^\ast(X)\subset\cdots H^\ast(X;\Q)
\] 
the perverse filtration associated with the map $h$. The associated graded pieces are denoted by
\[
\Gr^P_kH^\ast(X):=P_kH^\ast(X)/P_{k-1}H^\ast(X).
\]

In \cite{dCM1}, de Cataldo and Migliorini proved that for any proper morphism between smooth complex algebraic varieties, the associated perverse filtration satisfies the relative hard Lefschetz property. More precisely,   
\begin{equation}\label{0002}
\Gr^P_{r-k}H^\ast(X)\xrightarrow{\cup\alpha^k}\Gr^P_{r+k}H^\ast(X), ~~~k\ge0
\end{equation}
is an isomorphism for any ample $2$-form $\alpha$ on $X$.
 
\subsection{The $P=W$ identities} 
Given any compact complex algebraic variety $X$ and a reductive group $G$, there are two natural moduli spaces associated with them, the $G$-character variety $\mathcal{M}_B$ and the moduli of Higgs $G$-bundles $\mathcal{M}_{D}$. The Higgs moduli $\mathcal{M}_D$ is properly fibered over an affine space $h:\mathcal{M}_D\to\mathbb{A}$. Simpson proved in \cite{Si} that there is a canonical diffeomorphism between $\mathcal{M}_B$ and $\mathcal{M}_{D}$. A remarkable phenomenon discovered by de Cataldo, Hausel and Migliorini in \cite{dCHM} asserts that under the identification induced by the Simpson's diffeomorphism
\[
H^\ast(\mathcal{M}_{D};\Q)\xrightarrow{=}H^\ast(\mathcal{M}_{B};\Q),
\]
the perverse filtration on $H^\ast(\mathcal{M}_{D};\Q)$ associated with $h$ is expected to match the halved mixed Hodge-theoretic weight filtration on $H^\ast(\mathcal{M}_{B};\Q)$ , i.e.  
\begin{equation}\label{0001}
P_kH^\ast(\mathcal{M}_{D})= W_{2k}H^\ast(\mathcal{M}_B)=W_{2k+1}H^\ast(\mathcal{M}_B),k\ge0.
\end{equation}
The equality (\ref{0001}) is referred to as the $P=W$ identity. Under the $P=W$ identity, the relative hard Lefschetz property (\ref{0002}) on $H^\ast(M_D;\Q)$ pulls back to the following identity on $H^\ast(M_B,\Q)$
\begin{equation}\label{0003}
\Gr^W_{d-2k}H^\ast(\mathcal{M}_B)\xrightarrow[\cong]{\cup\beta^k}\Gr^W_{d+2k}H^\ast(\mathcal{M}_B),
\end{equation}
where $d=\dim_\C\mathcal{M}_B$ and $\beta\in H^2(\mathcal{M}_B;\Q)$ is of weight $4$. Satisfying the identity (\ref{0003}) is referred to as the curious hard Lefschetz property; see \cite{HV}. The curious hard Lefschetz property is a necessary condition to expect a variety to be the $W$ side in a $P=W$ identity. Unlike the relative hard Lefschetz property holds for any proper morphism between complex algebraic varieties, the curious hard Lefschetz property is rarely observed for algebraic varieties. Most of the known examples are related to character varieties.

In \cite{LS}, Lam and Speyer proved that cluster varieties of Louise type satisfy the curious Hard Lefschetz property, which suggests that cluster varieties could be the $W$ sides to produce new $P=W$ identities. In this paper, we will focus on the 2 dimensional cluster varieties and prove $P=W$ identities (\ref{0001}) for them. The corresponding $P$ sides are deformation equivalence classes of certain elliptic fibrations. Due to the lack of moduli interpretations, we don't expect a canonical diffeomorphism between them. However, we will prove that under $\it{any}$ diffeomorphism between them, the $P=W$ identity holds. The existence of a diffeormophism is proved by the theory of 4-manifold surgery. 
\begin{theorem}[Theorem \ref{main}]
Let $X$ be a 2 dimensional cluster variety. Then there exists a deformation of elliptic fibrations such that for any elliptic fibration $h:Y\to \Delta$ in the deformation, the following holds.
\begin{enumerate}
\item The underlying $C^\infty$ manifolds of $X$ and $Y$ are diffeomorphic.
\item For any diffeomorphism $f:X\to Y$, the $P=W$ identity holds, i.e. under the pullback isomorphism $f^*:H^*(Y;\Q)= H^*(X;\Q)$, 
\[
P_kH^*(Y;\Q)=W_{2k}H^*(X;\Q)=W_{2k+1}H^*(X;\Q)
\]
for any $k\ge0$.
\end{enumerate}
\end{theorem}

\medskip In \cite{dCHM}, the $P=W$ identity is proved between (twisted) $G$-character variety $\mathcal{M}_B$ of and the moduli of degree 1 $G$-Higgs bundles $\mathcal{M}_D$ over a proper curve of genus at least 2, for $G=\mathrm{GL}(2,\C),\mathrm{SL}(2,\C)$ and $\mathrm{PGL}(2,\C)$. In recent progress \cite{S,SZ}, the $P=W$ identity is proved for five families of parabolic moduli spaces. The $W$ sides are parabolic character varieties on a punctured curve and the $P$ sides are moduli of parabolic Higgs bundles. The Hitchin map $h:\mathcal{M}_{D}\to \mathbb{A}$ are (Hilbert schemes of) elliptic fibrations of the types $I_0$, $I_0^\ast$, $IV^\ast$, $III^\ast$ and $II^\ast$ according to the Kodaira's classification of elliptic fibrations. In \cite{Z}, a numeric version of the $P=W$ identity is proved for certain 2 dimensional moduli spaces of wild parabolic Higgs bundles of rank $2$ over $\P^1$. The $P=W$ identities for cluster varieties studied in this paper turn out to be new because the elliptic fibrations on the $P$ side are obstructed to be defined on $\A^1$; they are defined on the unit disk only. Since the singular fibers in our elliptic fibrations occur in the local study of the moduli wild parabolic Higgs bundles over $\P^1$ with 2 marked points \cite{I}, our result suggests that cluster varieties are analytically local version of character varieties. We do not know how to interpret this observation precisely, but we do expect a deep relation between them.

\subsection*{Acknowledgements} The author thanks Mark de Cataldo, Thomas Lam, Jun Li, Jingchen Niu, Junliang Shen, David Speyer, Ruijie Yang and Boyu Zhang for helpful discussions.

\section{2D cluster varieties and weight filtrations}
In this section, we study the (co)homology groups and the mixed Hodge structures of 2 dimensional cluster varieties. They will be on the $W$ side in the $P=W$ identities. We refer to \cite[Section 4]{LS} for notations and definitions for cluster varieties. In particular, we work with cluster varieties whose exchange matrices have skew-symmetric principal parts. 2-dimensional cluster varieties are classified by the number of frozen variables as follows.
\begin{enumerate}
\item Type 0. No mutable variable and two frozen variables. Then we obtain the cluster torus $\C^\ast\times\C^\ast$. Denote it by $X_0$
\item Type I. One mutable variable and one frozen variable. Then the exchange matrix has the form $\left(\begin{array}{c}0\\b\end{array}\right)$, where $b$ is an positive integer. We denote the corresponding cluster variety $_IX_b$. 
\item Type II. Two mutable variables and no frozen variable. Then the exchange matrix has the form $\left(\begin{array}{cc}0&b\\-b&0\end{array}\right)$, where $b$ is an positive integer. We deonte the corresponding cluster variety $_{II}X_b$.
\end{enumerate}
We will omit subscripts when no confusion arises.

\subsection{Type I} The cluster variety of type $_IX_b$ is studied in \cite[Section 7]{LS}. 
\begin{proposition}{\cite[Proposition 7.1]{LS}} \label{hodge}
Let $_IX_b$ be the 2D cluster variety defined by the exchange matrix
\[
\left(\begin{array}{c}0\\b\end{array}\right),
\]
where $b$ is a positive integer. Then 
\begin{enumerate}
\item $_IX_b$ is the affine surface 
\[
_IX_b=\{(x,x',y)\in\C^3\mid xx'=y^b+1,y\neq0\}.
\]
\item $_IX_b$ deformation retracts to the locus $|x|=|x'|$, $|y|=1$, which is $b$ 2-spheres arranging as a necklace. We denote them by $S_1,\cdots,S_b$.
\item $H_1({_IX_b};\Q)=\Q$, and $H_2({_IX_b};\Q)=\Q^{b}$, where generators in $H_2({_IX_b};\Q)$ are represented by $S_1,\cdots,S_b$. 
\item The generator of $H^1({_IX_b};\Q)$ is $(2\pi i)^{-1}d\log y$. The generators of $H^2({_IX_b};\C)$ are represented by differential forms 
$\gamma,y\gamma,\cdots, y^{b-1}\gamma$, where $\gamma$ is the unique regular extension of $(2\pi i)^{-2}d\log x\wedge d\log y$. (See Remark \ref{2000}.)
\item The mixed Hodge structure on $H^\ast({_IX_b})$ is split over $\Q$. The dimensions of graded pieces are listed in the table below.  
\begin{center}
\begin{tabular}{cccccc}
\hline
 & $\Gr_0^W$ & $\Gr_1^W$ & $\Gr_2^W$ & $\Gr_3^W$ & $\Gr_4^W$\\
\hline
$H^0$ & $1$ & $0$ & $0$ & $0$ & $0$ \\
\hline
$H^1$ & $0$ & $0$ & $1$ & $0$ & $0$ \\
\hline
$H^2$ & $0$ & $0$ & $b-1$ & $0$ & $1$ \\
\hline
\end{tabular}
\end{center}
More precisely, the forms $(2\pi i)y\gamma,\cdots, (2\pi i)y^{b-1}\gamma$ are in $H^2({_IX_b};\Q(e^{\pi i/b}))$. The $(b-1)$-dimensional $\C$-vector space they span has a $\Q$-basis, and their $\Q$-span is $W_2H^2({_IX_b};\Q)$. The weight 4 piece $\Gr_4^WH^2({_IX_b};\Q)$ is generated by $\gamma$.
\end{enumerate}
\end{proposition}

\begin{remark} \label {2000}
When $x\ne0$, $\gamma=(2\pi i)^{-2}d\log x\wedge d\log y$. When $x'\ne0$, $\gamma=-(2\pi i)^{-2}d\log x\wedge d\log y$. They agree on the overlap since $x'dx+xdx'=by^{b-1}$ implies $d\log x\wedge d\log y+d\log x'\wedge d\log y=0$. Then $\gamma$ extends to $X_b$ by Hartogs' theorem because the locus $x=x'=0$ is of codimension 2.
\end{remark}

For later use, we need the following statements in the proof of \cite[Theorem 7.1]{LS}. 

\begin{proposition} \label{toy lemma}
For $j=1,\cdots, b$, and $c=1,\cdots, b-1$, we have 
\[
\int_{S_j}\gamma=\frac{1}{b}
\]
and
\[
\int_{S_j}y^c\gamma=\frac{1}{2\pi ic}(e^{(2j+1)c\pi i/b}-e^{(2j-1)c\pi i/b}).
\]
\end{proposition}

\subsection{Type II}
The geometry, cohomology and mixed Hodge structures of $_{II}X_b$ are described in the following proposition.

\begin{proposition} \label{hodge2}
Let $_{II}X_b$ be the 2D cluster variety defined by the exchange matrix
\[
\left(\begin{array}{cc}0&b\\-b&0\end{array}\right),
\]
where $b$ is a positive integer. We have the following.
\begin{enumerate}
\item $_{II}X_b$ is an affine surface obtained by gluing two copies of $_{I}X_b$ along $X_0$. More precisely, we have
\[
_{II}X_b=\{(x,y,x',y')\mid xx'=y^b+1, yy'=x^b+1\}\subset \C^4.
\]
It is covered by open sets $U_b=\{(x,y,x',y')\in {_{II}X_b}\mid y\ne0\}\cong{_IX_b}$ and $V_b=\{(x,y,x',y')\in {_{II}X_b}\mid x\ne0\}\cong{_IX_b}$, whose overlap is exactly $\{(x,y,x',y')\in{_{II}X_b}\mid x\ne0,y\ne0\}\cong X_0$.
\item The locus of $|x|=|x'|,|y|=1$  is $b$ 2-spheres $S_1,\cdots,S_b$ arranging as a necklace. The locus of $|y|=|y'|,|x|=1$ is also $b$ spheres $T_1,\cdots,T_b$ arranging as a necklace. Then $H_2({_{II}X_b},\Q)$ is generated by $S_1,\cdots,S_b,T_1,\cdots, T_b$ with relation $S_1+\cdots+S_b=T_1+\cdots+T_b$. 
\item $H^2({_{II}X_b};\Q)$ is generated by $\gamma,x\gamma,\cdots,x^{b-1}\gamma,y\gamma,\cdots,y^{b-1}\gamma$, where $\gamma$ is the unique extension of $(2\pi i)^{-2}d\log x\wedge d\log y$.
\item The mixed Hodge numbers of $_{II}X_b$ are
\begin{center}
\begin{tabular}{cccccc}
\hline
 & $\Gr_0^W$ &$ \Gr_1^W$ & $\Gr_2^W $&$ \Gr_3^W$ & $\Gr_4^W$\\
\hline
$H^0$ & $1$ & $0$ & $0$ & $0$ & $0$ \\
\hline
$H^1$ & $0$ & $0$ & $0$ & $0$ & $0$\\
\hline
$H^2$ & $0$ & $0$ & $2b-2$ & $0$ & $1$\\
\hline
\end{tabular}
\end{center}
More precisely, $\gamma$ generates $\Gr_4^WH^2(X)$, where $\gamma$ is the unique extension of $(2\pi i)^{-2}d\log x\wedge d\log y$. The subspace spanned by $x\gamma,\cdots,x^{b-1}\gamma,y\gamma,\cdots,y^{b-1}\gamma$ in $H^2(_{II}X_b;\C)$ has a $\Q$-basis, and their $\Q$-span is $W_2H^2(_{II}X_b;\Q).$ 
\end{enumerate}
\end{proposition}

\begin{proof}
Since the graph associated with the exchange matrix is acyclic, $_{II}X_b$ is isomorphic to its lower cluster variety $\Spec \C[x,x',y,y']/(xx'-y^b-1,yy'-x^b-1)$; see \cite[Corollary 1.21]{BFZ}. 
Statements (2), (3) and (4) follow from Mayer-Vietoris arguments for the cover $_{II}X_b=U_b\cup V_b$ and Proposition \ref{hodge}.
\end{proof}

\section{Elliptic surfaces and perverse filtrations}
In this section, we study the perverse filtrations associated with proper fibration of an surface over the unit disk. They will be the $P$ sides in the $P=W$ identities. To use the perverse decomposition theorem \cite[Theorem 0.6]{Sa}, we will always assume the surface to be K\"ahler. 

\subsection{Type I} We give a candidate for the $P$ side which matches $_{I}X_b$ in the $P=W$ identity. Let $h:{_IY_b}\to \Delta$ be the elliptic fibration on the disk of type $I_b$, where $b$ is an positive integer. General fibers of $h$ are elliptic curves and the central fiber $h^{-1}(0)$ is $b$ 2-spheres arranged in a necklace. The monodromy matrix is 
\[
\left(
\begin{array}{cc}
1&b\\
0&1
\end{array}
\right).
\]
We refer to \cite[Section V]{BHPV} for more details. Denote $\Delta^\ast=\Delta\setminus \{0\}$ the punctured unit disk and $j:\Delta^\ast\hookrightarrow \Delta$ the open embedding. 

\begin{proposition} \label{perverse}
There is a perverse decomposition
\begin{equation}\label{0100}
Rh_*\Q_{_IY_b}[1]\cong (\Q_\Delta[1])\bigoplus (j_*L[1]\oplus \Q_0^{\oplus b-1})[-1]\bigoplus (\Q_\Delta[1])[-2],
\end{equation}
where $\Q_0$ is the skyscraper sheaf of rank 1 supported at 0, and $L$ is the $\Q$-local system of rank $2$ on $\Delta^\ast$ with the monodromy $\left(\begin{array}{cc}1&b\\0&1\end{array}\right)$. In particular, the perverse numbers are 
\begin{center}
\begin{tabular}{cccc}
\hline
 & $\Gr_0^P$ & $\Gr_1^P$ & $\Gr_2^P$\\
\hline
$H^0$ & $1$ & $0$ & $0$ \\
\hline
$H^1$ & $0$ & $1$ & $0$ \\
\hline
$H^2$ & $0$ & $b-1$ & $1$ \\
\hline
\end{tabular}
\end{center}
Furthermore, $P_1H^2(_IY_b)$ is spanned by the fundamental classes of the irreducible components of the central fiber $h^{-1}(0)$ and
\begin{equation}\label{0101}
P_1H^2(_IY_b)=\im \{H^2_c(_IY_b)\to H^2(_IY_b)\}.
\end{equation}
\end{proposition}

\begin{proof}
The perverse decomposition (\ref{0100}) follows from \cite[Theorem 0.6]{Sa} and \cite[Theorem 3.2.3]{dCM}. The monodromy of the local system $L$ follows from the property of elliptic fibration of type $I_b$. To compute the perverse numbers, we have the exact sequence
\begin{equation}\label{0009}
0\to j_*L\to Rj_*L\to R^1j_*L[-1]\to 0.
\end{equation}
By checking the stalks, we have $R^1j_*L=\Q_0$. \v Cech cohomology argument shows that $H^2(\Delta,j_*L)=0$, and $H^i(\Delta,Rj_*L)=H^i(\Delta^\ast,L)=\Q$ for $i=0,1$.
 So the associated long exact sequence of (\ref{0009}) is 
\[
\begin{array}{ccccccc}
0&\to&H^0(\Delta, j_*L)&\to&\Q&\to&0\\
 &\to & H^1(\Delta, j_*L)&\to &\Q&\to&\Q\\
 &\to&0& & & & \\
 \end{array}
\]
We conclude that $H^0(\Delta,j_*L)\cong \Q$ and $H^1(\Delta,j_*L)=0$. To see that $P_1H^2(Y)$ is spanned by the irreducible components of the singular fiber $h^{-1}(0)$, it suffices to note that the cohomology of the direct summand $\Q_0^{\oplus b-1}$ computes the fundamental classes of irreducible components of the singular fiber. It has dimension $b-1$ since the central fiber $h^{-1}(0)$, the sum of fundamental classes of all irreducible components, is cohomologous to $0$. To prove (\ref{0101}), it suffices to study the image of
$
\mathbb{H}^2_c(\Delta,Rh_*\Q_{_IY_b})\to\mathbb{H}^2(\Delta,Rh_*\Q_{_IY_b}).
$
The decomposition (\ref{0100}) breaks this map as the direct sum of the follow three maps
\begin{align}
&H^1_c(\Delta,j_*L)\to H^1(\Delta,j_*L),\label{0102}\\
&H^0_c(\Q_0^{\oplus b-1})\to H^0(\Q_0^{\oplus b-1}),\label{0103}\\
&H^0_c(\Delta,\Q_\Delta)\to H^0(\Delta,\Q_\Delta)\label{0104}.
\end{align}
Since (\ref{0102}) and (\ref{0104}) are zero maps, and (\ref{0103}) is an identity, we see that the image of $H^2_c({_IY_b})\to H^2({_IY_b})$ is exactly spanned by the fundamental classes of irreducible components of the central fiber, which is $P_1H^2({_IY_b})$. 
\end{proof}

\begin{remark}
A natural question is whether the elliptic fibration of type $I_b$ can be realized as ambient complex manifolds of a proper morphism between algebraic varieties. The answer is no. In fact, it follows from the decomposition theorem \cite{BBD} that the perverse sheaves in the perverse decomposition between algebraic varieties must be semisimple. However, the summand $j_*L[1]$ is not. 
\end{remark}

\subsection{Type II} We give a candidate for the $P$ side which matches $_{II}X_b$ in the $P=W$ identity. Let $h:{_{II}Y_b}\to\Delta$ be an elliptic fibration over unit disk with 2 singular fibers $h^{-1}(p_1)$ and $h^{-1}(p_2)$ both of type $I_b$, whose vanishing cycles generate the $H^1$ of a general fiber. The same argument as Proposition \ref{perverse} leads to the following proposition.

\begin{proposition} \label{perverse2}
  We have a perverse decomposition
\[
Rh_\ast\Q_Y[1]\cong(\Q_\Delta[1])\bigoplus(j_\ast L[1]\oplus\Q_{p_1}^{\oplus b-1}\oplus \Q_{p_2}^{\oplus b-1})[-1]\bigoplus(\Q_\Delta[1])[-2],
\]
where $L$ is the rank 2 $\Q$-local system on $\Delta\setminus\{p_1,p_2\}$, whose monodromy matrices around $p_1$ and $p_2$ are $\left(
\begin{array}{cc}
1&b\\
0&1
\end{array}
\right)
$ and
$
\left(
\begin{array}{cc}
1&0\\
-b&1
\end{array}
\right)
$, respectively. In particular, the perverse numbers are
\begin{center}
\begin{tabular}{cccc}
\hline 
           & $\Gr_0^P$ & $\Gr_1^P$ & $\Gr_2^P$ \\
\hline
$H^0$ & $1$            &    $0$         &$0$            \\
\hline
$H^1$ & $0$            &   $0$          &$0$\\
\hline
$H^2$ & $0$            &   $2b-2$     & $1$ \\
\hline
\end{tabular}
\end{center}
Furthermore, $P_1H^2({_{II}Y_b})$ is spanned by the fundamental classes of the irreducible components of the singular fibers $h^{-1}(p_1)$ and $h^{-1}(p_2)$. 
\end{proposition}

\section{Construction of diffeomorphisms}
On the contrary to the $P=W$ identities studied in \cite{dCHM}, \cite{S}, \cite{SZ} and \cite{Z}, we don't have moduli interpretations for cluster varieties and elliptic fibrations in Section 2 and 3. So we need to construct diffeomorphisms between them to establish the $P=W$ identities. The constructions are based on handlebody decompositions. As we will see in Section 5, the $P=W$ identities do not depend the choice of the diffeomorphism, which is not obvious a priori.

\subsection{Type I} In this section, we construct a diffeomorphism $f:{_IX_b}\to{_IY_b}$.  We will realize the cluster variety $_IX_b$ as the total space of a Lefschetz $T^2$-fibration over $\R^2$ which shares the same combinatorial information as the corresponding elliptic fibration $h:{_IY_b}\to\Delta$. Monodromy invariants and the covering trick will produce a desired diffeomorphism between $_IX_b$ and $_IY_b$. 
\begin{proposition}  \label{diffeo}
Let $_IX_b$ be the surface $xx'=y^b+1$ in $\C^2\times\C^\ast$ as in Section 2.1. Consider the smooth map $g:{_IX_b}\to\R^2$ defined as follows.
$$
\begin{array}{rclc}
g:&{_IX_b}&\to& \R^2\\
 &(x,x',y)&\mapsto& (|x|^2-|x'|^2,\log|y|).
\end{array}
$$
The fiber over $(0,0)$ is a necklace consists of $b$ 2-spheres and all other fibers are 2-tori. 
\end{proposition}

\begin{proof}
Denote $X_{s,t}=g^{-1}(s,t)$ the fiber of $g$ over point $(s,t)\in \R^2$. Then $X_{s,t}$ is cut out by equations $|x|^2-|x'|^2=s$ and $|y|=e^t$. To study $X_{s,t}$, we consider map 

$$
\begin{array}{rccc}
p_{s,t}:& X_{s,t}&\to& \C^\ast\\
 & (x,x',y)&\mapsto & y 
\end{array}
$$

\begin{enumerate}
\item When $s=t=0$, $X_{0,0}$ is cut out by equations $|x|=|x'|$ and $|y|=1$. So the image of $p_{0,0}$ is the unit circle. When $y$ is not a $b$-th root of $-1$, we have 
\[
\begin{cases}
x=|\sqrt{y^b+1}|e^{i\theta}\\
x'=\displaystyle\frac{y^b+1}{x},
\end{cases}
\]
so $p_{0,0}^{-1}(y)$ is a circle parametrized by $\theta$. When $y$ is a $b$-th root of $-1$, we have $x=x'=0$. Therefore, $X_{0,0}$ is obtained by gluing $b$ 2-spheres in a circle. 
\item When $(s,t)\ne(0,0)$, the image $p_{s,t}(X_{s,t})$ is a circle $|y|=e^t$. We claim that for any $y$ on the circle $|y|=e^t$, the fiber $p_{s,t}^{-1}(y)$ is also a circle. In fact, we have $|x|^2-|x'|^2=s$ and $|xx'|^2=|y^b+1|^2$, so 
\[
|x|=\sqrt{\frac{s+\sqrt{s^2+4|y^b+1|^2}}{2}}
\]
and
\[
|x'|=\sqrt{\frac{-s+\sqrt{s^2+4|y^b+1|^2}}{2}}.
\]
Note that when $(s,t)\neq(0,0)$, at least one of $|x|,|x'|$ is nonzero, and that $x$ and $x'$ determine each other since $y$ is fixed. So we get $p_{s,t}^{-1}(y)$ is a circle. Therefore $X_{s,t}$ is a topological circle bundle over a circle, hence is a torus.
\end{enumerate}
\end{proof}

The existence of a diffeomphism between $_IX_b$ and $_IY_b$ follows from handlebody decompositions on 4-manifolds. We refer to \cite{GS,HKK,KM} for definitions of Lefschetz fibrations, handlebody decompositions and plumbing constructions.

\begin{theorem}\label{diff}
There exists a diffeomorphism $f:{_IX_b}\to {_IY_b}$.
\end{theorem}

\begin{proof}
We first prove the case when $b=1$. It follows from Proposition \ref{diffeo} that both $g:{_IX_1}\to \R^2$ and $h:{_IY_1}\to\Delta$ are Lefschetz $T^2$-fibrations with one critical point on the central fiber. Up to diffeomorphism, such fibration is unique. Indeed, it is obtained by attaching a $2$-handle to the trivial $T^2$-fibration $T^2\times \Delta\to\Delta$ along the vanishing cycle on any boundary fiber with framing $-1$. See \cite[Section 8.2]{GS} and \cite[Page 242]{KM} for detailed treatments. So there is a diffeomorphism $f:{_IX_b}\to {_IY_b}$. For general $b$, $_IX_b=\{xx'=y^b+1,y\ne0\}$ is a $b$-fold \'etale covering of $_IX_1$, given by $y\mapsto y^b$.  Note that the fundamental group $\pi_1(X_b)=\Z$, and the covering $X_b\to X_1$ corresponds to the subgroup $b\Z\subset\Z$. $_IY_b$ is obtained as plumbing together $b$ copies of cotangent disk bundles of 2-spheres according to the graph of a cycle with $b$ nodes. In particular, $_IY_1$ is obtained as the self-plumbing of one cotangent disk bundle of a 2-sphere. It follows from the construction that $_IY_b$ is a smooth $b$-fold covering of $Y_1$, which corresponds to the subgroup $b\Z\subset\Z=\pi_1(_IY_1)$. Now by the classification of covering spaces, there exists a diffeomorphism $f:{_IX_b}\to {_IY_b}$ as desired.
\end{proof}

\subsection{Type II} In this section, we construct a diffeomorphism $f:{_{II}X_b}\to {_{II}Y_b}$.
\begin{proposition}\label{diff2}
There exists a diffeomorphism $f:{_{II}X_b}\to{_{II}Y_b}$.
\end{proposition}

\begin{proof}
By Proposition \ref{hodge2}.(1), the cluster variety $_{II}X_b\in\C^4$ is obtained by gluing $\{xx'=y^b+1,y\ne0\}$ and $\{yy'=x^b+1,x\ne0\}$ along the locus $\{x\ne0,y\ne0\}$. We cover $_{II}X_b$ by three open sets:
\begin{align*}
&X^\circ={_{II}X_b}\cap\{x\ne0,y\ne0\},\\
&U_\epsilon={_{II}X_b}\cap\{|x|<\epsilon,|y^b+1|<\epsilon\},\\
&V_\eta={_{II}X_b}\cap\{|x^b+1|<\eta, |y|<\eta\}.
\end{align*} 
Then for small $\epsilon$ and $\eta$, $U_\epsilon$ and $V_\eta$ are both diffeomorphic to $b$ disjoint copies of $\R^4$, and $U_\epsilon\cap V_\eta=\phi$. Let $\pi:X^\circ\rightarrow \R^2$ be the trivial elliptic fibration defined by $(x,y)\mapsto (\log|x|,\log|y|)$. The intersection $U_\epsilon^\circ=X^\circ\cap U_\epsilon$ can be described in terms of the fibration $\pi$ as  
\[
\pi(U_\epsilon^\circ)=\left\{(s,t)\in\R^2\mid s<\log\epsilon,\frac{\log(1-\epsilon)}{b}<t<\frac{\log(1+\epsilon)}{b}\right\},
\]
and for such $(s,t)$, $U_\epsilon^\circ\cap\pi^{-1}(s,t)$ is consists of $b$ copies of $S^1\times \R^1$. More precisely, let $u=\arg x,~~~v=\arg y$ be a parametrization of the fiber $\pi^{-1}(s,t)$, then 
\[
U_\epsilon\cap\pi^{-1}(s,t)=\left\{(u,v)\in T^2\mid \frac{(2c+1)\pi-\theta_{s,t,\epsilon}}{b}<v<\frac{(2c+1)\pi+\theta_{s,t,\epsilon}}{b},0\le c<b\right\}
\]
where $\theta_{s,t,\epsilon}=\arccos~~~(1+e^{2t}-\epsilon^2)/2e^t$. Note that if we let $\epsilon\to 0$, then $s\to-\infty$, $t\to0$, and $\theta_{s,t,\epsilon}\to0$. In the language of attaching handle, gluing $U_\epsilon$ onto $X^\circ$ is exactly attaching $b$ 2-handles along the $b$ circles 
\[
\left\{(u,v)\in T^2|v=\frac{(2c+1)\pi}{b},0\le c\le b-1\right\},
\]
on the boundary fiber over $(-\infty,0)$. Note that $X^\circ\cup U_\epsilon={_IX_b}$. So the framing numbers at the $b$ attaching circles have to be $-1$. We apply the same argument to attach $V_\eta$ to $X^\circ$, which is equivalent to attach $b$ 2-handles along the $b$ circles
\[
\left\{(u,v)\in T^2\mid u=\frac{(2c+1)\pi}{b},0\le c\le b-1\right\}
\]
on the boundary fiber over $(0,-\infty)$ with framing numbers -1. Therefore, $_{II}X_b$ is diffeomorphic to $_{II}Y_b$, the total space of an elliptic fibration over open disk with two $I_b$ singular fibers, whose vanishing cycles form a basis of $H_1(T^2,\Z)$, where $T^2$ is a general fiber.
\end{proof}

\section{The $P=W$ identities}
\setcounter{subsection}{-1}
\subsection{Type 0} Let $h:Y_0\to \Delta$ be the natural projection $\textrm{pr}_2:E\times\Delta\to \Delta$, where $E$ is any elliptic curve. Then $X_0=\C^*\times\C^*$ and $Y_0$ are both diffeomorphic to $\R^1\times\R^1\times S^1\times S^1$, hence are diffeomorphic. The $P=W$ identity follows from the study of $\tilde{A}_0$ case in \cite[Section 5]{Z}. 
\begin{proposition}\label{thm0}
For any diffeomorphism $f:X_0\to Y_0$, the $P=W$ identity (\ref{0001}) holds.
\end{proposition}

\subsection{Type I} In this section, we show that for the cluster variety $_IX_b$ (Section 2.1) and the elliptic fibration $h:{_IY_b}\to\Delta$ (Section 3.2), any diffeomorphism between them induces a $P=W$ identity.
\begin{proposition}\label{w}
Let $Z$ be the locus cut out by $|x|=|x'|$ and $|y|=1$ in $_IX_b$. Then $W_2H^2({_IX_b})$ is the kernel of the linear functional
\[
\int_Z:H^2({_IX_b})\to \Q.
\]
\end{proposition}

\begin{proof} 
By Proposition \ref{hodge}.(5), $W_2H^2({_IX_b})\otimes\C$ is generated by $\C$-basis $y\gamma,\cdots, y^{b-1}\gamma$.  
By Proposition \ref{toy lemma}, we have
\[
\int_Z y^c\gamma=\sum_{j=1}^b\int_{S_j}y^c\gamma=\sum_{j=1}^b\frac{1}{2\pi ic}(e^{(2j+1)c\pi i/b}-e^{(2j-1)c\pi i/b})=0
\]
for $c=1,\cdots, b-1$ and
\[
\int_Z \gamma=\sum_{j=1}^b\int_{S_j}\gamma=1.
\]
We conclude that $W_2H^2({_IX_b})\otimes\C$ is exactly the kernel of the linear functional 
\[
\int_Z:H^2(X;\C)\to\C.
\]
Since $Z$ is a rational homology class, the kernel of $\int_Z$ admits rational basis, and hence the conclusion follows from the following lemma.
\end{proof}

\begin{lemma}
Let $V$ be a $\Q$-vector space and let $W_1,W_2$ be two subspaces of $V$. Suppose that $W_1\otimes\C$ and $W_2\otimes\C$ are identical as subspaces of $V\otimes\C$, then $W_1=W_2$ as subspaces of $V$. 
\end{lemma}

\begin{proposition} \label{p}
Let $h:{_IY_b}\to \Delta$ be the elliptic fibration of type $I_b$. Let $F=h^{-1}(0)$ be the singular central fiber, viewed as a homology class. Then $P_1H^2({_IY_b})$ is the kernel of the linear functional 
\[
\int_F:H^2({_IY_b})\to \Q.
\]
\end{proposition}

\begin{proof}
When $b=1$, $P_1H^2({_IY_1})=0$. $H^2(Y)$ is spanned by a section class $\sigma$ with $\int_F \sigma=1$. Hence the result follows. When $b>1$, the self-intersections of all irreducible components of $F$ are $-2$ and the intersection numbers between adjacent components are $1$. It follows that all exceptional curves in central fiber $F$ are in the kernel of the pairing. By Proposition \ref{perverse}, $P_1H^2({_IY_b})$ is exactly the $(b-1)$-dimensional vector space spanned by the classes of exceptional divisors. Since the pairing $H_2({_IY_b})\times H^2({_IY_b})\to \Q$ is non-degenerate, the kernel of $\int_F$ is at most of dimension $b-1$. Therefore the kernel of $\int_F$ is exactly of dimension $b-1$ and coincides with $P_1H^2({_IY_b})$.
\end{proof}

\begin{theorem}\label{thm}
For any diffeomorphism $f:{_IX_b}\to {_IY_b}$, the $P=W$ identity (\ref{0001}) holds.
\end{theorem}

\begin{proof}
By comparing the mixed Hodge numbers in Proposition \ref{hodge} and perverse numbers in Proposition \ref{perverse}, it suffices to identify $W_2H^2({_IX_b})$ and $P_1H^2({_IY_b})$ via the diffeomorphism. Since $Z$ and $F$ are deformation retracts of $_IX_b$ and $_IY_b$ respectively, the homology class $Z$ in $H_2({_IX_b},\Q)$ is sent to $F$ or $-F$ in $H_2({_IY_b};\Q)$ via any diffeomorphism. Now by the Proposition \ref{w} and Proposition \ref{p}, $P_1H^2({_IY_b})=W_2H^2({_IX_b})$ and hence the $P=W$ identity holds.
\end{proof}

\begin{corollary}\label{5.6}
\[
W_2H^2({_IX_b})=\im \{H^2_c({_IX_b})\to H^2({_IX_b})\}.
\]
\end{corollary}

\begin{proof}
This identity follows from Proposition \ref{perverse} and Theorem \ref{thm}.
\end{proof}

\subsection{Type II}
In this section, we prove the $P=W$ identity for $_{II}X_b$ and $_{II}Y_b$ via any diffeomorphism.
\begin{proposition} \label{w2}
\[
W_2H^2(_{II}X_b)=\im\{H_c^2(_{II}X_b)\to H^2(_{II}X_b)\}.
\]
\end{proposition}

\begin{proof}
By Proposition \ref{hodge2}.(1), $_{II}X_b=U_b\cup V_b$, where $U_b\cong V_b\cong{_IX_b}$ and $U_b\cap V_b=X_0$. Denote $i_U:U_b\to {_{II}X_b}$ and $i_V:V_b\to {_{II}X_b}$.
We first show that the diagram 
\[
\begin{tikzcd}
H^2_c({_{II}X_b})\arrow[d,"\iota"]& H^2_c(U_b)\oplus H^2_c(V_b)\arrow[l,"i_*"']\arrow[d,"\iota'"]\\
H^2({_{II}X_b})\arrow[r,"i^*"] & H^2(U_b)\oplus H^2(V_b)
\end{tikzcd}
\]
commutes, where $i_*(\alpha,\beta)=i_{U*}\alpha-i_{V*}\beta$ and $i^*\alpha=(i^*_U\alpha,-i^*_V\alpha)$. By the symmetry between $U$ and $V$, it suffices to show that a basis of $H^2_c(U_b)$ maps $0$ in $H^2(V_b)$ via $i^*\circ \iota\circ i_*$. To see this, by Proposition \ref{hodge}.(3) and Proposition \ref{hodge2}.(2), $H^2_c(U_b)\cong H_2(U_b)$ is generated by spheres $S_1,\cdots,S_b$. Their images under $i^*\circ\iota\circ i_*$ are the Borel-Moore classes represented by $S_i\cap V_b$. They are the boundaries of the Borel-Moore 3-classes in $V_b$ represented by 
\[
\{(x,y,x',y')\in V_b\mid 0<|x|<\sqrt{|y^b+1|}, |y|=1,(2\pi i-1)/b<\arg y<(2\pi i+1)/b\},
\] 
hence are 0 in $H^2(V_b)$. This proves that the diagram commutes. Since $H^1_c(U_b\cap V_b)\cong H_c^1(X_0)=0$, the compactly supported Mayer-Vietoris sequence shows that $i_*$ is surjective. Similarly, $H^1(U_b)\oplus H^1(V_b)\cong H^1(U_b\cap V_b)$ implies that $i^*$ is injective. It follows that $i^*$ induces an isomorphism $\im \iota\xrightarrow{\sim}\im\iota'$. Now the identity follows from Corollary \ref{5.6} and
\[
W_2H^2({_{II}X_b})\xrightarrow{\sim}W_2H^2(U_b)\oplus W_2H^2(V_b).
\]
\end{proof}

\begin{theorem} \label{thm2}
For any diffeomorphism $f:{_{II}X_b}\to{_{II}Y_b}$, the $P=W$ identity (\ref{0001}) holds.
\end{theorem}

\begin{proof}
By the perverse numbers and mixed Hodge numbers in Proposition \ref{hodge2} and Proposition \ref{perverse2}, it suffices to prove 
\[
P_1H^2({_{II}Y_b})=W_2H^2({_{II}X_b}).
\]
Proposition \ref{hodge2} implies that $P_1H^2({_{II}Y_b})=\im \{H^2_c({_{II}Y_b})\to H^2_c({_{II}Y_b})\}$, and Proposition \ref{w2} implies that $W_2H^2({_{II}X_b})=\im \{H^2_c({_{II}X_b})\to H^2_c({_{II}X_b})\}.$ This completes the proof.
\end{proof}

\subsection{Deformation invariance of the perverse filtrations}
\begin{definition}\label{1010}
A deformation of elliptic fibrations is a commutative diagram of K\"ahler manifolds
\[
\begin{tikzcd} 
Y\arrow[rr,"h"]\arrow[rd]& & D\arrow[ld,"v"]\\
 & S, & 
\end{tikzcd}
\]
where $\dim Y=\dim S+2$ and $\dim D=\dim S+1$, such that
\begin{enumerate}
\item $h$ is proper surjective. $v$ is surjective.
\item $v\circ h:Y\to S$ is a locally trivial $C^\infty$ fibration.
\item For any $s\in S$, the restriction $h_s:(v\circ h)^{-1}(s)\to v^{-1}(s)$ is an elliptic fibration.
\end{enumerate} 
Two elliptic fibrations $h_1:Y_1\to D_1$ and $f_2:Y_2\to D_2$ are deformation equivalent if there exists a deformation of elliptic fibrations $(Y,D,S)$, and two points $s_1,s_2$ on $S$, such that $h_1=h_{s_1}$ and $h_2=h_{s_2}$.
\end{definition}

\begin{remark}
In this definition, the topological type of the fibration is allowed to be deformed. Only the underlying $C^\infty$ manifold of the total space is required to be fixed. For example, the elliptic fibration $h:{_{II}Y_b}\to \Delta$ is deformation equivalent to the elliptic fibration of type II. See \cite{N} for more examples.
\end{remark}

\begin{proposition} \label{def}
If two fibrations $h_1:Y_1\to D_1$ and $h_2:Y_2\to D_2$ are deformation equivalent, then there exists a diffeomorphism $t:Y_1\to Y_2$ such that the pull-back map 
\[
t^*: H^*(Y_2)\xrightarrow{\sim}H^*(Y_1)
\]
is an isomorphism of perverse filtrations, i.e.
\[
t^*:P_kH^*(Y_2)\xrightarrow{\sim}P_kH^*(Y_1).
\]
\end{proposition}

\begin{proof}
By Definition \ref{1010}.(2), the local $C^\infty$ triviality of $v\circ h$ produces a diffeomorphism $t:Y_1\to Y_2$. It also follows from the triviality that vanishing cycle $\phi\Q_Y=0$ and $R(v\circ h)_*\Q_Y$ has locally constant cohomology sheaves. Also note that $\Q_Y$ is simple. The proposition follows from \cite[Theorem 3.2.1.(iii)]{dCMa}. 
\end{proof}

Combining Proposition \ref{thm0}, Theorem \ref{thm}, Theorem \ref{thm2} and Proposition \ref{def}, we have the main theorem.
\begin{theorem} \label{main}
Let $X$ be a 2 dimensional cluster variety. Then there exists a deformation of elliptic fibrations such that for any elliptic fibration $h:Y\to D$ in the deformation, the following holds.
\begin{enumerate}
\item The underlying $C^\infty$ manifolds of $X$ and $Y$ are diffeomorphic.
\item For any diffeomorphism $f:X\to Y$, the $P=W$ identity holds, i.e. under the pullback isomorphism $f^*:H^*(Y;\Q)= H^*(X;\Q)$,
\[
P_kH^*(Y;\Q)=W_{2k}H^*(X;\Q)=W_{2k+1}H^*(X;\Q)
\]
for any $k\ge0$.
\end{enumerate}
More precisely, when $X=X_0,{_IX_b}$ or ${_{II}X_b}$, the corresponding elliptic fibration $h:Y\to D$ is deformation equivalent to $Y_0\to \Delta,{_IY_b}\to\Delta,{_{II}Y_b}\to\Delta$, respectively. 
\end{theorem} 

Since perverse filtrations encode topological information only, it is natural to ask whether the $P$ side is compeletely determined on the underlying $C^\infty$ manifold structure of $X$.
We make the following conjecture.

\begin{conjecture}
Let $X$ be a 2 dimensional cluster variety. Let $h:Y\to\Delta$ be a Lefschetz $T^2$-fibration over the unit disk, such that $Y$ is diffeomorphic to $X$. Then for any diffeomorphism $f:X\to Y$, the $P=W$ identity (\ref{0001}) holds.
\end{conjecture}

When the fibration $h:Y\to\Delta$ underlies a K\"ahler structure, the arguments in Proposition \ref{perverse} and Proposition \ref{perverse2} show that the perverse numbers are determined by the Betti numbers, and hence match the mixed Hodge numbers of $X$. It also follows that the peverses numbers are constant in a family of Lefschetz $T^2$-fibrations which vary smoothly (instead of holomorphically). More understanding on Kirby calculus and the behavior of perverse filtartions in non-complex setting will be helpful to relate any such $Y$ to $_IY_b$ or $_{II}Y_b$ to compare the perverse filtrations.

\end{document}